\documentclass[12pt]{amsart}

\usepackage{amssymb}

\usepackage{enumitem}

\usepackage{graphicx}

\makeatletter
\@namedef{subjclassname@2020}{%
  \textup{2020} Mathematics Subject Classification}
\makeatother

\usepackage[T1]{fontenc}

\newtheorem{theorem}{Theorem}[section]

\newtheorem{lemma}[theorem]{Lemma}

\theoremstyle{definition}

\newtheorem{remark}[theorem]{Remark}

\newtheorem{conjecture}[theorem]{\bf{Conjecture}}
\newtheorem{question}[theorem]{\bf {Question}}
\newtheorem{prop}{\bf {Proposition}}[section]

\numberwithin{equation}{section}
\newcommand{\N}{\mathbb{N}}
\newcommand{\Z}{\mathbb{Z}}
\newcommand{\C}{\mathbb{C}}

\begin{document}


\baselineskip=17pt


\title[]{The orbit intersection problem in positive characteristic}

\author[S. S. Rout]{Sudhansu Sekhar Rout}
\address{Department of Mathematics,
National Institute of Technology Calicut - 673 601, Kozhikode,  India.}
\email{sudhansu@nitc.ac.in}

\date{}

\begin{abstract}
In this paper, we study the orbit intersection problem for the linear space and the algebraic group in positive characteristic. Let $K$ be an algebraically closed field of positive characteristic and let $\Phi_1, \Phi_{2}: K^d \longrightarrow K^{d}$ be affine maps,  $\Phi_i({\bf x}) = A_i ({\bf x}) + {\bf x_i}$  (where each $A_i$ is a $d\times d$ matrix and ${\bf x}\in K^d$). If none of the eigenvalues of the matrices $A_i$ are roots of unity and each ${\bf a}_i \in K^d$ is not $\Phi_i$-preperiodic, then we prove that the set 
$\left \{(n_1, n_2) \in \Z^{2} \mid \Phi_1^{n_1}({\bf a}_1) =  \Phi_{2}^{n_{2}}({\bf a}_{2})\right\}$ is $p$-normal in $\mathbb{Z}^{2}$ of order at most $d$. Further, let $\Phi_1,  \Phi_{2}: \mathbb{G}_m^d \longrightarrow \mathbb{G}_m^d$ be regular self-maps and ${\bf a}_1,  {\bf a}_2\in \mathbb{G}_m^d(K)$. Let $\Phi_1^0$ and $\Phi_2^0$ be group endomorphisms of $\mathbb{G}_m^d$ and  ${\bf y}, {\bf z} \in \mathbb{G}_m^d(K)$ such that $\Phi_1({\bf x}) = \Phi_1^{0}({\bf x}) + {\bf y}$ and $\Phi_2({\bf x}) = \Phi_2^{0}({\bf x}) + {\bf z}$. We show, under some conditions on the roots of the minimal polynomial of $ \Phi_1^{0}$ and $\Phi_2^{0}$, that the set 
$ \{(n_1, n_{2}) \in \N_0^{2} \mid \Phi_1^{n_1}({\bf a}_1) =  \Phi_{2}^{n_{2}}({\bf a}_{2})\}$ (where ${\bf a}_1,  {\bf a}_2\in \mathbb{G}_m^d(K)$) is a finite union of singletons and one-parameter linear families. To do so, we use results on linear equations over multiplicative groups in positive characteristic and  some results on systems of polynomial-exponential equations.
\end{abstract}



\maketitle

\section{Introduction}

Let $\mathbb{N}$ denote the set of positive integers and $\mathbb{N}_0 := \mathbb{N}\cup \{0\}$. For a set $X$ endowed with self map $\Phi$ and for any $m \in \mathbb{N}_{0}$, we denote by $\Phi^{m}$ the $m$-th iteration $\Phi\circ  \cdots \circ \Phi$ of $\Phi$ with $\Phi^{0}$ denoting the identity map on $X$. If $x\in X$, we define the {\em forward orbit} $\mathrm{Orb}_{\Phi}^{+}(x) := \{\Phi^{m}(x)\mid  m \in \mathbb{N}_{0}\}$. Similarly, if the self map $\Phi$ is invertible, then the {\em backward orbit} is defined as the collection of inverse images $\mathrm{Orb}_{\Phi}^{-}(x) :=\cup_{n \geq 0} \Phi^{-n}(x)$. In this paper, we are dealing with some problems of dynamical Mordell-Lang conjecture in positive characteristics. 

\subsection{Overview in characteristic $0$} 

Let $C$ be a curve of genus $g\geq 2$ defined over a number field $K$ of characteristic $0$. The classical Mordell conjecture states that the curve $C$ has finitely many $K$- rational points. Further, Lang conjectured that if $V$ is a subvariety of a semiabelian variety $G$ defined over $\C$ and $\Gamma$ is a finitely generated subgroup of $G(\C)$, then $V(\C)\cap \Gamma$ is a union of at most finitely many translates of subgroups of $\Gamma$. This conjecture was proved by Faltings \cite{fal} for abelian varieties and Vojta \cite[Theorem 0.2]{voj} for semiabelian varieties.
Motivated by the classical Mordell-Lang conjecture, various authors ( \cite{den}, \cite[Question 7.1]{bell}, \cite[Conjecture 1.7]{gt9}) have proposed the following dynamical Mordell-Lang conjecture.
\begin{conjecture}\label{con1}
Let $X$ be a quasi-projective variety defined over a number field $K$ of characteristic $0$, let $Y$ be any subvariety of $X$, let $\alpha \in X(K)$ and let $\Phi$ be an endomorphism on $X$.  Then the set 
\begin{equation}\label{eq1}
\{m\in \mathbb{N}_{0}\mid  \Phi^{m}(\alpha)\in Y(K)\}
\end{equation}
is a finite union of arithmetic progressions.
\end{conjecture}

An arithmetic progression is a set of the form $\{mk+\ell: k\in \mathbb{N}_0\}$ for some $m, \ell \in \mathbb{N}_0$ and when $m=0$, this set is a singleton. Conjecture \ref{con1} implies the cyclic case in the classical Mordell-Lang conjecture when $X$ is a semiabelian variety and $\Phi$ is the translation by a point $x\in X(K)$. It is natural to ask for a generalization of dynamical Mordell-Lang conjecture to a statement which would contain as a special case the full statement of classical Mordell-Lang conjecture. The \emph{general dynamical Mordell-Lang conjecture} is stated as follows: Suppose $\Phi_1, \ldots, \Phi_r$ are commuting $K$-morphisms from $X$  to itself and $x\in X(K)$. Then the set 
\[\{(n_1, \ldots, n_r) \in \mathbb{N}_0^r \mid \Phi_1^{n_1} \circ \cdots \circ\Phi_r^{n_r}(x) \in V(K)\}\]
is a finite union of translates of subsemigroups of $\mathbb{N}_0^r$. 

There is an extensive literature proving various special cases of Conjecture \ref{con1}. For example, Ghioca and Tucker \cite[Theorem 1.8]{gt9} proved when $\Phi$ is any algebraic group endomorphism  of a semiabelian variety. The case when $\Phi$ is \'etale endomorphism of any quasi-projective variety is solved in \cite{bgt10}. Then Xie \cite[Theorem A]{xie} proved this for birational endomorphism of the affine plane. Ostafe and Sha obtained \cite{os} quantitative version of Conjecture \ref{con1} for polynomial morphisms of several special types when $X$ is the affine $m$-space and $Y$ is a hypersurface. For a detailed survey of recent works on Conjecture \ref{con1} (see \cite{bgt16}). Further, Ghioca et. al., \cite {gtz11} have obtained various results for the general dynamical Mordell-Lang problem when $X$ is a semiabelian variety and the self maps are endomorphisms satisfying certain technical conditions.  But,  not much is known even when we restrict to the following special case called the {\it orbit intersection problem} (see \cite[p.1264]{gn}): 
\begin{question}\label{que1}
Let $X$ be a variety over algebraically closed field $K$ of characteristic $0$, let $r \geq 2$ be a positive integer. For $1 \leq i \leq r$, let $\Phi_i$ be a $K$-morphism from $X$ to itself, and let $\alpha_i \in X(K)$  is not $\Phi_i$-preperiodic. When
can we conclude that the set 
\[S := \{(n_1, \ldots, n_r ) \in \N^r_0 \mid \Phi_1^{n_1}(\alpha_1) = \cdots = \Phi_r^{n_r}(\alpha_r)\}\]
is a finite union of sets of the form $\{(n_1k + \ell_1,\ldots,n_rk + \ell_r): k \in \N_0\}$ for some $n_1,\ldots ,n_r, \ell_1,\ldots ,\ell_r \in \N_0$?
\end{question}

Question \ref{que1} is known for the case when $X = \mathbb{P}^1_K$ and each $\Phi_i$ is a polynomial of degree larger than $1$  (see \cite[Theorem 1.1]{gtz08} and \cite[Theorem 1.1]{gtz12}). Later, Question \ref{que1} is answered when $X$ is a semiabelian variety \cite[Theorem 1.4]{gn} and when $X = \mathbb{A}_K^{n}$ and the self maps are affine transformations \cite[Theorem 1.6]{gn}. Further,  various upper bounds are derived for the orbit intersection problem when $X$ is an affine $n$-space and self maps are polynomial morphisms of special types (see  \cite[p.125-127]{r20}).
 
\subsection{Overview in positive characteristic} 

The picture in positive characteristic is very much different. From the following example, one can observe that the Mordell-Lang conjecture is not true in positive characteristic.  From now onwards, we let $p$ be a prime number and $q = p^e$ for some positive integer $e$. 

Let $K = \mathbb{F}_p(t)$ be the field of rational functions over the field of size $p$. Let $G:= \mathbb{G}_m^2$ be a semiabelian variety defined over $K$ and $V$ be the subvariety defined by the equation $x + y =1$. Let $\Gamma$ be  the subgroup of $G(K)$ generated by $(t, 1 - t)$. Then $V(K) \cap \Gamma = \{(t^{p^e}, (1-t)^{p^e})\mid e \in \mathbb{N}_0\}$ and this set cannot be expressed as a finite union of translates of subgroups of $\Gamma$.

Using model theoretic ideas, Hrushovski \cite[Theorem 1.1]{h96} proved the classical Mordell-Lang conjecture in positive characteristic. Moosa-Scanlon \cite{ms} proved a form of the classical Mordell-Lang conjecture for semiabelian varieties defined over finite fields. Motivated by the work of Moosa-Scanlon \cite{ms}, the following conjecture was proposed (see \cite[Conjecture 13.2.0.1]{bgt16}, \cite[p. 1152]{g19}). 

\begin{conjecture}[Dynamical Mordell-Lang Conjecture in postive characteristic]\label{con2}
Let $X$ be a quasi-projective variety defined over a field $K$ of characteristic $p$, let $Y$ be any subvariety of $X$ defined over $K$, let $\alpha \in X(K)$ and let $\Phi: X\longrightarrow X$ be an endomorphism defined over $K$.  Then the set $\{m\in \mathbb{N}_{0}\mid  \Phi^{(m)}(\alpha)\in Y(K)\}$ is a union of finitely many arithmetic progressions along with  finitely many sets of the form 
\begin{equation}\label{eq2}
\left\{\sum_{j=1}^{m}c_jp^{k_jn_j} \mid n_j\in \mathbb{N}_{0},\; j= 0,1, \ldots, m\right\},
\end{equation}
for some $m\in \mathbb{N}$, for some $c_j \in \mathbb{Q}$, and some $k_j \in \mathbb{N}_0$.
\end{conjecture}

As of now, Conjecture \ref{con2} is known in the following cases. If $X$ is a semiabelian variety defined over finite field and $\Phi$ is an algebraic group endomorphism whose action on the tangent space at the identity is given through a diagonalization matrix  (see \cite[Proposition 13.3.0.2]{bgt16}). Further, if $X = \mathbb{G}_m^N$ and $Y\subset X$ is a curve (see \cite[Theorem 1.3]{g19}) and if $X = \mathbb{G}_m^N, \;Y\subset X$ be a subvariety of dimension at most equal to $2$ and $\Phi$ is any regular self map (see \cite[Theorem 1.2]{c19}). Also, Conjecture \ref{con2} holds for any subvariety $Y\subset \mathbb{G}_m^N$ assuming $\Phi$ is an algebraic group endomorphism with the property that no iterate of it restricts to being a power of the Frobenius on the proper algebraic subgroup $\mathbb{G}_m^N$ (see \cite[Theorem 1.3]{c19}). Conjecture  \ref{con2} seems to be very difficult. To solve in case of arbitrary subvarieties of $X = \mathbb{G}_m^N$, it leads to difficult questions involving polynomial-exponential equations. Recently Ghioca et al., \cite{goss} obtained a general quantitative result in support of the case of endomorphisms of semiabelian varieties of Conjecture \ref{con2}.

In this paper, we study Question \ref{que1} in positive characteristic. From the following example (see \cite[p.1267]{gn}), one can see that Question \ref{que1} fails for affine maps defined over $\mathbb{F}_p(t)$. Let $\Phi_i : \mathbb{A}^1 \longrightarrow \mathbb{A}^1$ be affine maps defined by $\Phi_1(x) =  t(x -1) +1$ and $\Phi_2(x) = (t+1)x$. Now it is easy to see $\Phi_1^{n_1} (2)= t^{n_1} +1$ and $\Phi_2^{n_2} (1)= (t+1)^{n_2}$. Then the set 
\[\left\{(n_1, n_2)\in \mathbb{N}_0^2 \mid \Phi_1^{n_1}(2) = \Phi_2^{n_2} (1) \right\} = \left\{(p^n, p^n)\mid n \in \mathbb{N}_0\right\}\]
 is not a finite union of arithmetic progressions. This example motivate us to consider the orbit intersection problem in positive characteristic.  In order to state our theorems, we need the following definitions  (see \cite{der15}):

\smallskip

Suppose that $k \geq 1, {\bf c_0}, {\bf c_1}, \ldots, {\bf c_k} \in \mathbb{Q}^m$ with $(q - 1){\bf c_i} \in \mathbb{Z}^m$ for all $i$ , ${\bf c_0} + \cdots+ {\bf  c_k}\in \mathbb{Z}^m$. Then we define 
\[S_q({\bf c_0}; {\bf c_1}, \ldots, {\bf c_k} )= \left\{ {\bf c_0} +{\bf c_1}q^{f_1}+ \cdots+ {\bf  c_k}q^{f_k}\mid f_1, \ldots, f_k \in \mathbb{N}_0 \right\}.\]
The conditions on $ {\bf c_0}, {\bf c_1}, \ldots, {\bf c_k}$ imply that $S_q({\bf c_0}, {\bf c_1}, \ldots, {\bf c_k}) \subset \mathbb{Z}^m$. We call it an {\it elementary $p$-nested set} in $\mathbb{Z}^m$ of oder at most $k$. We define a {\it $p$-normal set} in $\mathbb{Z}^m$ of order at most $k$ as a finite union of singletons and sums $R + S$, where $R$ is a subgroup of $\mathbb{Z}^m$ and $S$ is either a singleton or an elementary $p$-nested set in $\mathbb{Z}^m$ of order at most $k$. For $m =1$, this definition was first introduced by Derksen \cite[p.177]{der07} to study the zero set of linear recurrence sequences in positive characteristic. Now we are ready to state our first result.

\begin{theorem}\label{th1}
Let $K$ be an algebraically closed field of characteristic $p$ and let $d\in \mathbb{N}$. Let $\Phi_1, \Phi_{2}: K^d \longrightarrow K^{d}$ be affine maps (that is there exist a $d\times d$ matrix $A_i \in GL_d(K)$ and ${\bf x_i}\in K^d$ such that $\Phi_i({\bf x}) = A_i ({\bf x}) + {\bf x_i}$ for $i = 1, 2)$. Let ${\bf a}_i \in K^d$ be not $\Phi_i$-preperiodic for $i = 1, 2$. If neither $A_1$ nor $A_2$ have an eigenvalue which is a root of unity, then the set 
\begin{equation}\label{eq3}
\left \{(n_1, n_2) \in \Z^{2} \mid \Phi_1^{n_1}({\bf a}_1) =  \Phi_{2}^{n_{2}}({\bf a}_{2})\right\}
\end{equation}
is $p$-normal in $\mathbb{Z}^{2}$ of order at most $d$.
\end{theorem}
 
\begin{remark}
In Theorem \ref{th1},  the eigenvalues of the matrices $A_1$ and $A_2$ do not belong to $\overline{\mathbb{F}}_p^*$.
\end{remark}
 To state the second main result, we need the following notation. An one-parameter linear family in $\Z^2$ is of the form $(x(t), y(t)) = (at + c,  bt + d), \; t\in \Z$ for some integers $a, b \in \Z\setminus\{0\}$ and  $c, d \in \Z$.  We get the one-parameter linear  family in $\N_0^2$ by intersecting $\N_0^2$ with one-parameter linear family in $\Z^2$ and removing some singletons.

\begin{theorem}\label{th2}
Let $K$ be an algebraically closed field of characteristic $p$ and $d\in \mathbb{N}$. Let $\Phi_1,  \Phi_{2}: \mathbb{G}_m^d \longrightarrow \mathbb{G}_m^d$ be regular self-maps and ${\bf a}_1,  {\bf a}_2\in \mathbb{G}_m^d(K)$. Let $\Phi_1^0$ and $\Phi_2^0$ be group endomorphisms of $\mathbb{G}_m^d$ and  ${\bf y}, {\bf z} \in \mathbb{G}_m^d(K)$ such that $\Phi_1({\bf x}) = \Phi_1^{0}({\bf x}) + {\bf y}$ and $\Phi_2({\bf x}) = \Phi_2^{0}({\bf x}) + {\bf z}$. Let $\lambda_1, \ldots, \lambda_r$ and $\delta_1, \ldots, \delta_{s}$ be the roots of the minimal polynomial of $ \Phi_1^{0}$ and $\Phi_2^{0}$ respectively, such that none of $\lambda_i$'s $(1\leq i\leq r)$ and $\delta_j$'s $(1\leq j\leq s)$ are roots of unity. Further, assume that $|\lambda_1|>\max\{1, |\lambda_2|, \ldots, |\lambda_r|\}$ and $|\delta_1|>\max\{1, |\delta_2|, \ldots, |\delta_s|\}$. Then the set 
\begin{equation}\label{eq4a}
\left \{(n_1, n_{2}) \in \N_0^{2} \mid \Phi_1^{n_1}({\bf a}_1) =  \Phi_{2}^{n_{2}}({\bf a}_{2})\right\}
\end{equation}
is a finite union of singletons and one-parameter linear families.
\end{theorem}

\section{Proof of Theorem \ref{th1}}\label{sec2}
\subsection{Auxiliary results}
  Let $K$ be a field of positive characteristic $p$ and let $G$ be a finitely generated
subgroup of the multiplicative group $K^{*}$. For any subgroup $G$ of $K^{*}$ and a positive integer $n$ it makes sense to write $\mathbb{P}_n(G)$ for the set of points in projective space defined over $G$.  For $n\geq 2$, let $V$ be a linear variety in $\mathbb{P}_n$ defined over $K$. We write $V(G) = V \cap \mathbb{P}_n(G)$ for the set of points defined over $G$.

We will need the radical $G = \sqrt{G}$. For us this remains in $K$; thus it is the group of $\gamma$ in $K$ for which there exists a positive integer $s$ such that $\gamma^s$ lies in $G$. We denote by $\varphi = \varphi_q$ the Frobenius with $\varphi(x) =x^q$, where $q= p^e$ for some positive integer $e$. For points $\pi_0, \pi_1, \cdots, \pi_h$, we define the set
\begin{equation}
(\pi_0, \pi_1, \cdots, \pi_h) := (\pi_0, \pi_1, \cdots, \pi_h)_q = \pi_0 \bigcup_{e_1 =0}^{\infty}\cdots \bigcup_{e_h =0}^{\infty} (\varphi^{e_1}\pi_1)\cdots (\varphi^{e_h}\pi_h)
\end{equation}
with the interpretation $\pi_0$ itself if $h=0$. 

\begin{prop}[ \cite{der12}, Theorem 2]\label{prop1}
Let $K$ be a field of positive characteristic $p$, let $V$ be an arbitrary linear variety defined over $K$, and suppose that $\sqrt{G}$ in $K$ is finitely generated. Then there is a power $q$ of $p$ such that $V(G)$ is an effectively computable union of sets $(\pi_0, \pi_1, \cdots, \pi_h)_q H(G)$ with points $\pi_0, \pi_1, \cdots, \pi_h\; (0\leq h \leq n-1)$ defined over $\sqrt{G}$ and subgroups $H$.
\end{prop}

We also need nested sets in abelian groups. Let $\mathcal{A}$ be a finitely generated abelian group. For $k\geq1$, let $C_0, C_1, \ldots, C_k$ in $\mathcal{A}$, we define 
\[U_q(C_0; C_1, \ldots, C_k) = \{C_0 + C_1 q^{f_1}+ \cdots C_kq^{f_k} \mid f_1, \ldots, f_k \in \mathbb{N}_0\}\] in $\mathcal{A}$. We call it an {\em elementary integral $p$-nested set} of order at most $k$.

\begin{prop}[ \cite{der15}, p.119]\label{prop2}
Let $B_1, \ldots, B_d$ in $\mathcal{A}$, let $\mathcal{B}$ be a subgroup of $\mathcal{A}$ and denote by $R$ the subgroup of all $(k_1, \ldots, k_d) \in \mathbb{Z}^d$ such that $k_1B_1 + \cdots + k_dB_d$ lies in $\mathcal{B}$. Let $U$ be an elementary integral $p$-nested set of order at most $k$ in $\mathcal{A}$. Then the set of all $(k_1, \ldots, k_d)$ in $\mathbb{Z}^d$ such that $k_1B_1 + \cdots + k_dB_d$ lies in $\mathcal{B}+U$ is either empty or $R+ S$, where $S$ is a finite union of singletons and elementary $p$-nested sets of order at most $k$ in  $\mathbb{Z}^d$.
\end{prop}

By an application of Laurent's theorem \cite{la84}, we obtain the following result.

\begin{prop}\label{prop6}
The intersection of two elementary $p$-nested sets is a finite union of elementary $p$-nested sets.
\end{prop}
\subsection{Proof of Theorem \ref{th1}}

Let $\tilde{\bf a}_1$ be a fixed point of $\Phi_1$, that is, $A_1\tilde{\bf a}_1 + {\bf a}_1  = \tilde{\bf a}_1$. This is permissible since $A_1- I_d$ is invertible. Define $\psi({\bf a}) = {\bf a} + \tilde{\bf a}_1$, so that $\psi^{-1}\circ \Phi_1\circ \psi({\bf a} ) = A_1 ({\bf a} )$. From this one can deduce,  $\Phi_1^{n_1}({\bf a} + \tilde{\bf a}_1) = A_1^{n_1}({\bf a}) + \tilde{\bf a}_1$. Similarly, let  $\tilde{\bf a}_2$ be a fixed point of $\Phi_2$, then we have $\Phi_2^{n_2}({\bf a} + \tilde{\bf a}_2) = A_2^{n_2}({\bf a}) + \tilde{\bf a}_2$. Hence, we reduce the problem of studying the set of pairs $(n_1, n_2) \in \Z^2$ satisfying 
\begin{equation}\label{eq13}
A_1^{n_1}{\bf b_1} = A_2^{n_2}{\bf b_2} +  {\bf b_3}
\end{equation}
where ${\bf b_3}:= \tilde{\bf a}_2 - \tilde{\bf a}_1$ and ${\bf b_1}, {\bf b_2}$ are given vectors such that ${\bf b_1}$
 (respectively  ${\bf b_2}$) is not preperiodic under the map ${\bf x} \mapsto A_1{\bf x}$ (respectively ${\bf x} \mapsto A_2{\bf x}$).

Suppose that  $J_1$ and $J_2$ are Jordan normal form of matrix $A_1$ and $A_2$ respectively. Hence, we can write  $A_1 = C^{-1}J_1 C$ and $A_2 = D^{-1}J_2 D$, where $C, D \in GL_d(K)$. With these expressions of $A_1$ and $A_2$, equation \eqref{eq13} becomes
\begin{equation}\label{eq14}
J_1^{n_1}C{\bf b_1}= CD^{-1}J_2^{n_2}D{\bf b_2}+  C{\bf b_3}.
\end{equation}

Let $\lambda \in K, \ell\in \N$ and $n\in \N_0$. Let $J_{\lambda, \ell}$ be the Jordan matrix of size $\ell$ and eigenvalue $\lambda$ and we have the formula 
\begin{equation}\label{eq15}
J_{\lambda, \ell}^n = \begin{bmatrix}
\lambda^n & \binom{n}{1}\lambda^{n-1} & \binom{n}{2}\lambda^{n-2} &\cdots  & \binom{n}{\ell-1}\lambda^{n-\ell+1}\\
0 & \lambda^{n} & \binom{n}{1}\lambda^{n-1} &\cdots  & \binom{n}{\ell-2}\lambda^{n-\ell+2}\\
\vdots& \vdots& \vdots& \vdots & \vdots\\
\vdots& \vdots& \vdots& \vdots & \vdots\\
0& 0& 0&\cdots & \lambda^n
\end{bmatrix}.
\end{equation}

Since $\binom{n}{k}$ is a polynomial in $n$ of degree $k$, there is a power $Q$ of $p$ such that all $\binom{n}{k}$ in \eqref{eq15} depend only on the values of $n$ modulo $Q$. Thus for all $n$ in each fixed residue class $n_0+ Q\Z$ in $\Z$, we may write \eqref{eq15} as 
 
\begin{equation}\label{eq16}
J_{\lambda, \ell}^n = \begin{bmatrix}
\lambda^n & \gamma_{1,1}\lambda^{n} & \gamma_{1,2}\lambda^{n} &\cdots  & \gamma_{1,\ell-1}\lambda^{n}\\
0 & \lambda^{n} & \gamma_{2,1}\lambda^{n} &\cdots  & \gamma_{2, \ell-2}\lambda^{n}\\
\vdots& \vdots& \vdots& \vdots & \vdots\\
\vdots& \vdots& \vdots& \vdots & \vdots\\
0& 0& 0&\cdots & \lambda^n
\end{bmatrix}.
\end{equation}
    If ${\bf b} = (b_1, \ldots, b_{\ell})^{T}$, is a fixed column vector in $K^n$ $(^T$ denote the transpose) and $n\geq \ell$, we have
\begin{equation}\label{eq17}
J_{\lambda, \ell}^n{\bf b} = (\gamma_1\lambda^n, \ldots, \gamma_{\ell}\lambda^n)^{T}.
\end{equation}
Let $r$ be number of Jordan blocks in $J_1$, let $J_{\lambda_i, g_i}$ for $1\leq i\leq r, \lambda_i\in K, g_i\in \N$ and $\sum_{i}g_i=d$ be the Jordan blocks of $J_1$. Let $s$ be number of Jordan blocks in $J_2$, let $J_{\delta_i, h_i}$ for $1\leq i\leq s, \delta_i\in K, h_i\in \N$ and $\sum_{i}h_i=d$ be the Jordan blocks of $J_2$.

Assume  that $n_1$ (resp. $n_2$) is in fixed residue class $n_1'+ Q\Z$ (resp. $n_2'+ Q\Z$) in $\Z$. This is possible as for any $(n_1', n_2')\in \Z^2$ and $p$-normal set $R+S$, the set $(n_1', n_2') + Q (R + S) = \tilde{R} + \tilde{S}$ for the subgroup $\tilde{R} = QR$ and the elementary $p$-nested set $\tilde{S} = (n_1', n_2') + QS$. Thus, by \eqref{eq17}, 
\begin{equation}\label{eq18}
J_1^{n_1}C{\bf b}_1 = (b_{1,1}^{(1)}\lambda_1^{n_1}, \ldots b_{1,g_1}^{(1)}\lambda_1^{n_1}, b_{2,1}^{(1)}\lambda_2^{n_1}, \ldots b_{2,g_2}^{(1)}\lambda_2^{n_1},\cdots, b_{r,1}^{(1)}\lambda_r^{n_1}, \ldots b_{r,g_r}^{(1)}\lambda_r^{n_1})^{T},
\end{equation}
and
\begin{equation}\label{eq19}
J_2^{n_2}D{\bf b}_2 = (b_{1,1}^{(2)}\delta_1^{n_2}, \ldots b_{1,h_1}^{(2)}\delta_1^{n_2}, b_{2,1}^{(2)}\delta_2^{n_2}, \ldots b_{2,h_2}^{(2)}\delta_2^{n_2},\cdots, b_{s,1}^{(2)}\delta_s^{n_2}, \ldots b_{s,h_s}^{(2)}\delta_s^{n_2})^{T}.
\end{equation}

Write $C{\bf b}_3 = (b_{1}^{(3)},\ldots, b_d^{(3)})$ and let $k$-th row of the matrix $CD^{-1}$ is
\[(c_{1,1}^{k}, \ldots, c_{1, h_1}^{k}, c_{2, 1}^{k}, \ldots, c_{2,h_2}^{k}, \ldots, c_{s, 1}^{k}, \ldots, c_{s, h_s}^{k})^{T}.\]

For fix $i\in \{1, \ldots, r\}$ and $\ell\in \{1, \ldots, g_i\}$, our claim is that the set of $(n_1, n_2)\in \Z^2$ satisfying \eqref{eq14} is $R+S$, where $R$ is the subgroup of $\Z^2$ and $S$ is the finite union of singletons and elementary $p$-nested sets of order at most $d$ in $\Z^2$. If we assume this claim is true, then application of Proposition \ref{prop6} finishes the proof of of the theorem.

Now fix an $i\in \{1, \ldots, r\}$ and $\ell\in \{1, \ldots, g_i\}$. From \eqref{eq14}, \eqref{eq18} and \eqref{eq19}, we have
\begin{equation}\label{eq20}
b_{i,\ell}^{(1)} \lambda_i^{n_1} = \sum_{j=1}^{s}\left(\sum_{k=1}^{h_j}c_{j,k}^{g_{i-1}'+\ell}b_{j,k}^{(2)}\right)\delta_{j}^{n_2} + b_{g_{i-1}'+\ell}^{(3)}, 
\end{equation}
with $g_{i}' =\sum_{j=1}^{i}g_j$. We rewrite \eqref{eq20} as 
\begin{equation}\label{eq21}
b_{i,\ell}^{(1)} \lambda_i^{n_1} = \sum_{j=1}^{s}d_{i,j,\ell}\delta_{j}^{n_2} + b_{g_{i-1}'+\ell}^{(3)}
\end{equation}
where $d_{i,j,\ell}:= \sum_{k=1}^{h_j}c_{j,k}^{g_{i-1}'+\ell}b_{j,k}^{(2)}$.

We want to apply Proposition \ref{prop1} to the linear variety $V$ defined by the corresponding equation
\begin{equation}\label{eq22}
 \sum_{j=1}^{s}d_{i,j,\ell}X_j + (-b_{i,\ell}^{(1)}) X_{s+1}+ b_{g_{i-1}'+\ell}^{(3)}X_{s+2} = 0.
\end{equation}
Here we are in projective space $\mathbb{P}_{s+1}$ and also we work inside the field generated by $d_{i,j,\ell}, b_{i,\ell}^{(1)}, b_{g_{i-1}'+\ell}^{(3)}, \lambda_i, \delta_j$ over $\mathbb{F}_p$ and $G$ as the radical of the group generated by $ \lambda_i, \delta_j$. This $G$ is also finitely generated. Notice that \eqref{eq21} gives a point  $\eta$ on $V(G)$.  We may identify the group $G^{s+1}$ with $\mathbb{P}_{s+1}(G)$ and define an isomorphism $\log$ from these to a finitely generated additive abelian group $\mathcal {A}$. Then we have 
\begin{equation}\label{eq23}
\log \eta = n_1 B_1 + n_2 B_2
\end{equation}
where $B_2 = \log (\delta_1, \delta_2, \cdots, \delta_s,1) = \log (\delta_1\cdots \delta_s)$  and $B_1 = \log (1, 1, \cdots, 1,\lambda_i)= \log (\lambda_i)$.

By Proposition \ref{prop1}, $V(G)$ is a finite union of sets \[Z:= (\pi_0, \pi_1, \ldots, \pi_h)_q H(G)\quad (0\leq h \leq s)\] with points $\pi_0, \pi_1, \ldots, \pi_h$ defined over $G$, 
\begin{equation*}
 (\pi_0, \pi_1, \cdots, \pi_h)_q = \pi_0 \bigcup_{e_1 =0}^{\infty}\cdots \bigcup_{e_h =0}^{\infty} (\varphi^{e_1}\pi_1)\cdots (\varphi^{e_h}\pi_h)
\end{equation*}
and linear subgroups $H$ defined by the equation $X_i=X_j$.  Now a point $\pi_Z \in Z$ has 
\begin{equation}\label{eq24}
\log \pi_Z = C_0+ C_1 q^{f_1}+ \cdots+C_hq^{f_h} + B
\end{equation}
with $C_k = \log \pi_k$ for $k=0, 1, \ldots, h$ and $B =\log \sigma$ in $H(G)$. Let $\mathcal{B}= \log H(G)$. Thus, the set of all such $\pi_m$ such that $\log \pi_m$ forms a sum $\mathcal{B} + U$, where $U$ is an elementary integral $p$-nested set of order at most $h \leq s\leq d$ in $\mathcal{A}$. Now from \eqref{eq23} and \eqref{eq24} and Proposition \ref{prop2}, one can observe that for each $Z$, the set $(n_1, n_2) \in \Z^2$ is $R_Z+ S_Z$, with $R_Z$ is group of all $(n_1, n_2)$ in $\Z^2$ with $n_1B_1 + n_2B_2$ in $\mathcal{B}$ and $S_Z$ is a finite union of singletons and elementary $p$-nested sets in $\Z^2$ of order at most $d$. This completes the proof of the theorem. \qed

\section{Proof of Theorem \ref{th2}}

\subsection{Linear recurrence sequences}\label{sec3.1}
 We define linear recurrence sequences which are essential in our proof.

A linear recurrence sequence of order $k$ is a sequence $(U_{n})_{n \geq 0}$ satisfying a relation
\begin{equation}\label{eq4}
U_{n} = b_1U_{n-1} + \dots +b_kU_{n-k}
\end{equation}
with $b_1,\dots, b_k \in \C$ with $b_k\neq 0$ and $U_0,\dots,U_{k-1}$ are integers not all zero. The characteristic polynomial of $U_n$ is defined by
\begin{equation}\label{eq5}
f(x):= x^k - b_1x^{k-1}-\dots-b_k = \prod_{i =1}^{t}(x - \alpha_i)^{m_i}\in \C[X]
\end{equation}
where $\alpha_1,\dots,\alpha_t$ are distinct and $m_1,\dots, m_t$ are positive integers. Then as it is well-known (see for example \cite[Theorem C1]{st}) we have a representation of the form
\begin{equation}\label{eq6}
U_n=\sum_{i=0}^t P_i(n)\alpha_{i}^n \ \ \ \text{for all}\ n\geq 0,
\end{equation}
with $P_i(x)$ is a polynomial of degree $m_i-1$ $(i=1,\dots,t)$. If some root of the characteristic polynomial is a root of unity, let this root be $\alpha_0$, and $\alpha_1, \ldots, \alpha_t$ the other roots. If no root of the characteristic polynomial is a root of unity, let these roots be $\alpha_1, \ldots, \alpha_t$ and set $\alpha_0 = 1, P_0 = 0$. We call the sequence $\{U_n\}$ {\it simple} if $t=k$. The sequence $(U_{n})_{n \geq 0}$ is called {\it degenerate} if there are integers $i,j$ with $1\leq i< j\leq t$ such that $\alpha_i/\alpha_j$ is a root of unity; otherwise it is called {\it non-degenerate}. Every degenerate recurrence sequence $(U_n)_{n\geq 0}$ can split into finitely many subsequences $(U_{nM+ \ell})$ for given $M \in \mathbb{N}$ and for $\ell = 0, 1, \ldots, M-1$ such that each subsequence is either trivial or non-degenerate recurrence sequence. If some $\alpha_i/\alpha_j$ is a root of unity, say of order $M$, then for each $\ell = 0, \ldots, M-1$ we have 
\begin{equation}\label{eq7g}
U_{nM+\ell} = \sum_{i=0}^t P_i(nM+\ell)\alpha_{i}^{\ell}\left(\alpha_{i}^M\right)^{n}
\end{equation}
and we can rewrite \eqref{eq7g} for $U_{nM+\ell}$ by collecting powers of $\alpha_i^M$ which are equal and thus get a non-degenerate linear recurrence sequence. For more details on linear recurrence sequences we refer the reader to (\cite{sch, st}).

Let $V_m$ be another linear recurrence, written as
\begin{equation}\label{eq7a}
V_m = \sum_{i=0}^{t'} Q_i(m)\beta_i^m
\end{equation} 
where $\beta_0$ is a root of unity, $\beta_i/\beta_j$ for $i\neq j$ is not a root of unity and  the $Q_i$ are polynomials of degree $m_i-1$ for $0\leq i\leq t'$ and  $Q_i \neq 0$ for $1\leq i\leq t'$. 

The following proposition is essentially a reformulation of Laurent \cite{la87, la89}, and will be used in our proof.

\begin{prop}[\cite{ss}, Proposition 1, 2 and 3]\label{prop4}
Suppose $(U)_{n\geq 0}$ and $(V)_{m\geq 0}$ are non-degenerate linear recurrence sequences given by \eqref{eq6} and \eqref{eq7a} and are not of the form $P(n)\alpha_0^n$ and $Q(m)\beta_0^n$ respectively, where $\alpha_0, \beta_0$ are roots of unity. Further, assume that $|\alpha_1|>\max \{1, |\alpha_2|, \ldots, |\alpha|_t\}$ and $|\beta_1|>\max \{1, |\beta_2|, \ldots, |\beta_{t'}|\}$. Then the set of pairs of integers $(n, m)$ for which 
\begin{equation}\label{eq7f}
U_n = V_m
\end{equation}
lie in the union of a finite set with possibly a one-parameter linear family of solutions
\begin{equation}\label{eq11}
x(t) = at + a', \quad y(t) = bt + b', \quad t\in \Z,
\end{equation}
with certain $a', b'\in \Z, \; a, b \in \Z \setminus\{0\}$.
\end{prop}


\begin{lemma}\label{prop5}
The intersection of two one-parameter linear families is again a family of linear type.
\end{lemma}
\begin{proof}
Let $(a_1t+c_1, b_1t+ d_1)$ and $(a_2t+c_2, b_2t+ d_2)$ with $t\in \Z$ be two linear families. Here $(a_i, b_i)$ for $i=1, 2$ be the minimal with $\alpha_i^{a_i} = \beta_i^{b_i}$. Let $(a, b)$ be the least common multiple of $(a_i, b_i)$ with $\alpha_i^{a} = \beta_i^{b}$. If the intersection of the two linear families in non-empty and if $(n, m)$ lies in this intersection, then the intersection consists of the pairs $(at +n, bt+m), t\in \Z$. This completes the proof.
\end{proof}

\subsection{Proof of Theorem \ref{th2}}
Since $\Phi_1$ is a regular self map on  $\mathbb{G}_m^d$, then there exists a group endomorphism $\Phi_1^{0}: \mathbb{G}_m^d \longrightarrow \mathbb{G}_m^d$ and there exists ${\bf y}\in \mathbb{G}_m^d(K)$ such that for any ${\bf x}\in \mathbb{G}_m^d(K)$, we have $\Phi_1({\bf x}) = \Phi_1^{0}({\bf x}) + {\bf y}$. Further, $\Phi_1^{0}$ acts on $\mathbb{G}_m^d$ as follows:
\begin{equation*}
\Phi_1^{0}(x_1, \ldots, x_d) = \left(\prod_{k=1}^{d} x_k^{a_{1, k}}, \ldots , \prod_{k=1}^{d} x_k^{a_{d, k}}\right)
\end{equation*}
 for some integer $a_{i,j}$, and therefore $\Phi_1^0$ is integral over $\Z$, i.e., there exist integers $c_0, \ldots, c_{d-1}$ such that
 \begin{equation}\label{eq25}
  (\Phi_1^{0})^d({\bf x})+ c_{d-1} (\Phi_1^{0})^{d-1}({\bf x})+\cdots +c_1  \Phi_1^{0}({\bf x}) + c_0 {\bf x} = 0,
 \end{equation}
 for each ${\bf x}\in \mathbb{G}_m^d(K)$. Then as shown in \cite[Claim 4.2]{g19}, there exist linear recurrence sequences $(U_{i,n})_{n \in \N_0} \subset \Z$ for $1\leq i \leq d$ and $(V_{i,n})_{n \in \N_0} \subset \Z$ for $0\leq i \leq d-1$ such that for each $n_1\in \N_0$, we have 
 \begin{equation}\label{eq26}
 \Phi_1^{n_1}({\bf a}) = \sum_{i=1}^{d}U_{i,n_1} \left(\sum_{j=0}^{i-1} (\Phi_1^0)^{j}({\bf y})\right) + \sum_{i=0}^{d-1} V_{i,n_1}(\Phi_1^0)^{i}({\bf a}).
 \end{equation}
 Similarly, for the regular self map $\Phi_2$ on  $\mathbb{G}_m^d$ and $n_2\in \N_0$, we have 
 \begin{equation}\label{eq27}
 \Phi_2^{n_2}({\bf b}) = \sum_{i=1}^{d}U_{i,n_2}' \left(\sum_{j=0}^{i-1} (\Phi_2^0)^{j}({\bf z})\right) + \sum_{i=0}^{d-1} V_{i,n_2}'(\Phi_2^0)^{i}({\bf b}),
 \end{equation}
 where $\Phi_2^{0}$ is a group endomorphism on $\mathbb{G}_m^d,\; (U_{i,n}')_{n_2 \in \N_0} \subset \Z$ for $1\leq i \leq d$ and $(V_{i,n}')_{n \in \N_0} \subset \Z$ for $0\leq i \leq d-1$ and ${\bf z} \in \mathbb{G}_m^d(K)$. For each  $i = 1, \ldots, d$, we use the following notation 
 \begin{equation}\label{eq28}
 P_i := \sum_{j=0}^{i-1} (\Phi_1^0)^{j}({\bf y}) \;\; \mbox{and}\;\;  P_i' := \sum_{j=0}^{i-1} (\Phi_2^0)^{j}({\bf z}).
 \end{equation}
 Now using \eqref{eq26}, \eqref{eq27} and \eqref{eq28}, we rewrite the equation $\Phi_1^{n_1}({\bf a}) =  \Phi_2^{n_2}({\bf b})$ as 
\begin{equation}\label{eq29}
\sum_{i=1}^{d}U_{i,n_1}P_i + \sum_{i=0}^{d-1} V_{i,n_1}(\Phi_1^0)^{i}({\bf a}) = \sum_{i=1}^{d}U_{i,n_2}' P_i'+ \sum_{i=0}^{d-1} V_{i,n_2}'(\Phi_2^0)^{i}({\bf b}).
 \end{equation}
 
 Let $\Gamma$ be a finitely generated abelian group containing 
 \[(\Phi_1^0)^{j}({\bf y}), (\Phi_2^0)^{j}({\bf z}), (\Phi_1^0)^{j}({\bf a})\quad \mbox{and} \quad(\Phi_2^0)^{j}({\bf b})\] for each $j = 0, \ldots, d-1$. Since $\Gamma$ is finitely generated abelian group, we know that $\Gamma$ is isomorphic to a direct sum of a finite subgroup $\Gamma_0$ with a subgroup $\Gamma_1$ which is isomorphic to $\Z^r$ for some $r\in \N_0$. Let $\{Q_1, \ldots, Q_r\}$ be a $\Z$-basis for $\Gamma_1$. Then we write each \[P_i = P_{i,0} + \sum_{k=1}^r e_{i,k}Q_k, \; P_{i,0}\in \Gamma_0, \;\mbox{and each}\;\; e_{i,k}\in \Z\] for $i = 1, \ldots, d$ and also write each  
 \[(\Phi_1^0)^{i}({\bf a}) = T_{i,0} + \sum_{k=1}^r f_{i,k}Q_k,\;T_{i,0}\in \Gamma_0 \;\mbox{and each}\;\; f_{i,k}\in \Z.\] 
 
 Similarly, we write each $P_i' = P_{i,0}' + \sum_{k=1}^r e_{i,k}'Q_k$ for $i = 1, \ldots, d$, where $P_{i,0}'\in \Gamma_0$ and each $e_{i,k}'\in \Z$ and also write each  $(\Phi_2^0)^{i}({\bf b})$ for $i = 1, \ldots, d$ as $T_{i,0}' + \sum_{k=1}^r f_{i,k}'Q_k$ with $T_{i,0}'\in \Gamma_0$ and each $f_{i,k}'\in \Z$. Comparing the torsion and free part of \eqref{eq29}, we have
 \begin{equation}\label{eq30}
 \sum_{i=1}^{d}U_{i,n_1}P_{i_0} + \sum_{i=0}^{d-1} V_{i,n_1}T_{i,0} = \sum_{i=1}^{d}U_{i,n_2}' P_{i_0}'+ \sum_{i=0}^{d-1} V_{i,n_2}'T_{i,0}'
 \end{equation}
and 
\begin{align}\label{eq31}
\begin{split}
 \sum_{i=1}^{d}U_{i,n_1}\sum_{k=1}^r e_{i,k}Q_k + &\sum_{i=0}^{d-1} V_{i,n_1}\sum_{k=1}^r f_{i,k}Q_k \\
& = \sum_{i=1}^{d}U_{i,n_2}' \sum_{k=1}^r e_{i,k}'Q_k+ \sum_{i=0}^{d-1} V_{i,n_2}'\sum_{k=1}^r f_{i,k}'Q_k.
 \end{split}
 \end{align}
 
 Since each of the points $P_{i,0}, T_{i,0}, P_{i,0}'$ and $T_{i,0}'$ belongs to a finite group $\Gamma_0$ and so they all have finite order bounded by $|\Gamma_0|$. Thus $\{U_{i,n_1} P_{i,0}\}_{n_1}$is preperiodic for each $i$. 
(Here we recall that a sequence $(W_n)_{n}$ is {\em preperiodic} if there exists a positive integer $\ell$ such that the subsequence $(W_n)_{n\geq \ell}$ is periodic). Therefore, the set of $(n_1, n_2)\in \N_0^2$ which satisfy \eqref{eq30} is a finite union of  sets of the form 
\[\{u_1m_1 + v_1, u_2m_2 +v_2: m_1, m_2\in \N_0\}\]
for some $u_1, u_2, v_1, v_2\in \N_0$.

Using Proposition $\ref{prop5}$, it suffices to prove that the set of all $(n_1, n_2)$ satisfying \eqref{eq31} is a finite union of linear families and singletons. Now comparing coefficient of each $Q_k$ in both sides of \eqref{eq31} for $k =1, \ldots, r$, we have 
\begin{equation}\label{eq32}
 \sum_{i=1}^{d}e_{i,k}U_{i,n_1} + \sum_{i=0}^{d-1} f_{i,k}V_{i,n_1} = \sum_{i=1}^{d}e_{i,k}'U_{i,n_2}' +  \sum_{i=0}^{d-1} f_{i,k}'V_{i,n_2}'.
\end{equation}
Since linear combination of linear recurrence sequence is a linear recurrence sequence (whose characteristic roots are among the characteristic roots of the original linear recurrence sequences), both sides of \eqref{eq32} are again linear recurrence sequences. Now \eqref{eq31} is equivalent to the simultaneous solutions of the equations 
\begin{equation}\label{eq33}
W_{k,n_1} = W_{k,n_2}' \quad \mbox{for each}\;\; k =1, \ldots, r,
\end{equation}
 where  $W_{k,n_1}: = \sum_{i=1}^{d}e_{i,k}U_{i,n_1} + \sum_{i=0}^{d-1} f_{i,k}V_{i,n_1}$ and $W_{k,n_2}' := \sum_{i=1}^{d}e_{i,k}'U_{i,n_2}' +  \sum_{i=0}^{d-1} f_{i,k}'V_{i,n_2}'$. Since $(W_{k,n_1})_{n_1\geq 0}$ is a linear recurrence sequence, by replacing $(W_{k,n_1})_{n_1\geq 0}$ by finitely many linear recurrence sequences, (which can be obtained by replacing $n_1$ by a suitable arithmetic progression $Mn_1 +\ell$), we may assume that the sequence $(W_{k,n_1})_{n_1\geq 0}$ is non-degenerate (see Section \ref{sec3.1}). Similarly, we can also assume $W_{k,n_2}'$ is non-degenerate.  Now we will discuss about the set 
 \begin{equation}\label{eq34a}
 Z:= \{(n_1, n_2) \in \N_0^2\mid W_{k,n_1} = W_{k,n_2}' \quad \mbox{for each}\;\; k =1, \ldots, r\}.
 \end{equation}
 Further, the characteristic roots of the linear recurrence sequences $V_{i, n_1}$ are among the roots of the minimal polynomial of $\Phi_1^0$, while the characteristic roots of $U_{i, n_1}$ are contained in the set consisting of $1$ and all the roots of the minimal polynomial of $\Phi_1^0$ (see \cite[Equation (14) and (17)]{g19}). Since by our assumption none of the roots of minimal polynomial of $\Phi_1^0$ are roots of unity, so the linear recurrence sequence $W_{k, n_1}$ is not of the form $P(n_1)\lambda_0^{n_1}$, where $\lambda_0$ is a root of unity. Similarly, $W_{k, n_2}$ is not of the form $Q(n_2)\delta_0^{n_2}$, where $\delta_0$ is a root of unity. Further, since $|\lambda_1|>\max\{1, |\lambda_2|, \ldots, |\lambda_r|\}$ and $\{|\delta_1|>\max\{1, |\delta_2|, \ldots, |\delta_s|\}$, then by Proposition \ref{prop4}, $Z$ is a finite union of singletons and a one-parameter linear familiy of solutions 
 \[(n_1, n_2) = (ut +u', vt+v')\cap \N_0^2,\quad  t\in \Z,\]
 with certain $u, v \in \Z \setminus \{0\}$ and $u', v'\in \Z$.
 


Thus, for each fixed $k\in \{1, \ldots, r\},\; Z$ is a finite union of singletons and a one-parameter linear family. Now applying Lemma \ref{prop5} finishes the proof of Theorem \ref{th2}. 
 
 \section{Concluding Remark}
 In the proof of Theorem \ref{th2},  one may notice that \eqref{eq4a} reduces to \eqref{eq34a}. If 
 we do not assume that none of the roots of the minimal polynomial of $\Phi_1^{0}$ and $\Phi_2^{0}$ are roots of unity in Theorem \ref{th2}, then there is a possibility that \eqref{eq34a} takes the form $P(n_1)\alpha^{n_1} = Q(n_2)\beta^{n_2}$, where $\alpha$ and $\beta$ are roots of unity and in such situation we will not get the solutions in required form. For example, if $W_{k, n_1} = n_1$ and $W_{k, n_2} = n_2^2$, then the set 
$Z$ in \eqref{eq34a} is $\{(n_1^2, n_1)\mid n_1\in \N\}$ which is not a linear family.

Further, if we do not assume \[|\eta_1|>\max\{1, |\eta_2|, \ldots, |\eta_s|\}, \mbox{and}\;\; \{|\delta_1|>\max\{1, |\delta_2|, \ldots, |\delta_t|\}\] in Theorem \ref{th2}, then the set \eqref{eq4a} lie in a finite union of singletons, one-parameter linear families and exponential families.

\subsection*{Acknowledgements}
We thank referee for her/his useful remarks.


\end{document}